\documentclass{amsart}
\usepackage{verbatim}
\usepackage[textsize=scriptsize]{todonotes}
\usepackage{tikz-cd}

\usepackage{etoolbox}

\usepackage{amsmath}
\usepackage{amssymb}
\usepackage{amsthm}
\usepackage{amscd}
\usepackage{enumerate}
\usepackage[pdfusetitle,unicode,hidelinks]{hyperref}
\usepackage{bbm}
\usepackage{etoolbox}

\usepackage[utf8]{inputenc}

\newtheorem{proposition}{Proposition}
\newtheorem{corollary}[proposition]{Corollary}
\newtheorem{lemma}[proposition]{Lemma}
\newtheorem{theorem}[proposition]{Theorem}

\newtheorem*{conjecture*}{Conjecture}
\newtheorem*{theorem*}{Theorem}
\newtheorem*{corollary*}{Corollary}
\newtheorem*{proposition*}{Proposition}
\newtheorem*{lemma*}{Lemma}
\theoremstyle{definition}
\newtheorem{definition}[proposition]{Definition}

\newtheorem{convention}[proposition]{Convention}
\newtheorem*{definition*}{Definition}
\newtheorem*{construction*}{Construction}
\theoremstyle{remark}
\newtheorem{remark}[proposition]{Remark}
\newtheorem*{remark*}{Remark}
\newtheorem{question}[proposition]{Question}
\newtheorem{example}[proposition]{Example}
\newtheorem*{example*}{Example}

\newcommand{\id}{\operatorname{id}}
\newcommand{\Z}{\mathbb{Z}}
\def\C{\mathbb C}
\newcommand{\N}{\mathbb{N}}
\newcommand{\Q}{\mathbb{Q}}

\let\scr=\mathcal
\let\bb=\mathbb

\def\P{\bb P}
\def\R{\bb R}
\newcommand{\1}{\mathbbm{1}}

\newcommand{\eff}{{\text{eff}}}

\newcommand{\SH}{\mathcal{SH}}

\DeclareMathOperator*{\colim}{colim}

\let\lim=\relax
\DeclareMathOperator*{\lim}{lim}
\def\Map{\mathrm{Map}}

\def\NAlg{\mathrm{NAlg}}

\def\Span{\mathrm{Span}}

\def\Spc{\mathcal{S}\mathrm{pc}{}}

\def\Fun{\mathrm{Fun}}

\newcommand{\wequi}{\simeq}
\newcommand{\Mod}{\text{-}\mathcal{M}od}

\DeclareRobustCommand{\ul}{\underline}

\newcommand{\tr}{\mathrm{tr}}

\let\cat=\mathrm
\def\Sm{{\cat{S}\mathrm{m}}}

\def\FEt{\mathrm{FEt}{}}

\def\mot{\mathrm{mot}}

\def\ph{\mathord-}

\usepackage{etoolbox}
\newtoggle{draft}
\togglefalse{draft}
\iftoggle{draft} {
\usepackage[margin=1.5in]{geometry}

\newcommand{\NB}[1]{\todo[color=gray!40]{#1}}
\newcommand{\TODO}[1]{\todo[color=red]{#1}}

\newcommand{\tombubble}[1]{\todo[color=green!40]{#1}}
\newcommand{\tom}[1]{{\color{green!60!black}#1}}
\newcommand{\jeremy}[1]{{\color{blue!60!black}#1}}
\newcommand{\jeremybubble}[1]{\todo[color=blue!40]{#1}}
}{ % else
\usepackage[margin=1in]{geometry}

\newcommand{\NB}[1]{}
\newcommand{\TODO}[1]{}
\newcommand{\tombubble}[1]{}
\newcommand{\tom}[1]{}
\newcommand{\jeremy}[1]{}
\newcommand{\jeremybubble}[1]{}
\renewcommand{\todo}[1]{}
}
\def\MGL{\mathrm{MGL}}
\def\MU{\mathrm{MU}}
\def\BP{\mathrm{BP}}
\def\BPGL{\mathrm{BPGL}}
\def\H{\mathrm{H}}
\def\Tr{\mathrm{Tr}}
\def\Res{\mathrm{Res}}
\def\N{\mathrm{N}}
\def\B{\mathrm{B}}
\def\free{\mathrm{free}}
\def\imap{\ul{\mathrm{map}}}
\def\EE{\mathbb{E}}
\def\fet{\mathrm{fet}}
\def\all{\mathrm{all}}
\def\cd{\mathrm{cd}}
\newcommand{\fpsr}[1]{[\![ #1 ]\!]}

\numberwithin{proposition}{section}

\title{Nilpotence in normed $\MGL$-modules}
\date{\today}

\author{Tom Bachmann}
\address{Department of Mathematics, Massachusetts Institute of Technology, Cambridge, MA, USA}
\email{tom.bachmann@zoho.com}

\author{Jeremy Hahn}
\address{Department of Mathematics, Massachusetts Institute of Technology, Cambridge, MA, USA}
\email{jhahn01@mit.edu}

\begin{document}
\maketitle

\begin{abstract}
We establish a motivic version of the May Nilpotence Conjecture: if $E \in \NAlg(\SH(S))$ satisfies $E \wedge \H\Z \wequi 0$, then also $E \wedge \MGL \wequi 0$.
In words, motivic homology detects vanishing of normed modules over the algebraic cobordism spectrum.
\end{abstract}

\tableofcontents

\section{Introduction}
A guiding principle of stable homotopy theory is the celebrated Nilpotence Theorem of Devinatz, Hopkins, and Smith \cite{devinatz1988nilpotence}:

\begin{theorem*}[DHS Nilpotence]
Suppose $R$ is a homotopy ring spectrum.  Then an element $x \in \pi_*(R)$ is nilpotent if and only if its Hurewicz image in $MU_*(R)$ is nilpotent.
\end{theorem*}

When $R$ is highly structured, a stronger nilpotence criterion, due to Mathew, Noel, and Naumann, is available \cite{mathew2017nilpotence}:

\begin{theorem*}[May Nilpotence Conjecture]
Suppose $R$ is an $\scr E_\infty$-ring spectrum.  Then an element $x \in \pi_*(R)$ is nilpotent if and only if its Hurewicz image in $\H\Z_*(R)$ is nilpotent.
\end{theorem*}

The proof of the May Nilpotence Conjecture in \cite{mathew2017nilpotence} runs by reduction to the DHS Nilpotence Theorem.  As such, the authors prefer to phrase the May Nilpotence Conjecture in the following equivalent form:

\begin{theorem*}[May Nilpotence Conjecture]
Suppose that $R$ is an $\scr E_\infty$-ring spectrum.  Then $R \wedge \MU \simeq 0$ if and only if $R \wedge \H\Z \simeq 0$.
\end{theorem*}

\begin{remark}
To see that the first version of the May Nilpotence Conjecture implies the second, note that $R \simeq 0$ if and only if $1$ is nilpotent in $\pi_*(R)$.  To see that the second version implies the first, note that $x$ is nilpotent in $\pi_*(R)$ if and only if $R[x^{-1}] \simeq 0$. 
Here $R[x^{-1}]$ denotes the mapping telescope (sequential colimit), which one may show is also an $\scr E_\infty$-ring.
\end{remark}

The present paper is part of a larger project to understand which nilpotence theorems continue to hold in the stable motivic category $\SH(S)$, where $S$ is a scheme.
While $\SH(S)$ is a symmetric monoidal stable $\infty$-category, and hence admits a notion of $\scr E_\infty$-ring object, there is additionally the stronger notion of a \emph{normed spectrum}, which in the notation of \cite{bachmann-norms} is an object of $\NAlg(\SH(S))$.
Many of the most important $\scr E_\infty$-rings in the classical stable homotopy category, such as $\MU, \text{KU},$ and $\H\Z$, have normed analogs, such as $\MGL,\mathrm{KGL},$ and $\H\Z \in \NAlg(\SH(S))$.
Here, we will be particularly interested in $\H\Z$, which denotes Spitzweck's motivic cohomology spectrum \cite{spitzweck2012commutative}, and $\MGL$, which denotes the algebraic cobordism spectrum (see, e.g., \cite[\S16]{bachmann-norms}).

Our main result is the following normed analog of the May Nilpotence Conjecture:
\begin{theorem} \label{thm:mainII}
Let $S$ be a noetherian scheme of finite dimension, and write $\scr S$ for the set of primes not invertible on $S$.
Let $E \in \NAlg(\SH(S))$ and suppose that \[ E \wedge \H\Z[\scr S^{-1}] \wequi 0. \]
Then also \[ E \wedge \MGL[\scr S^{-1}] \wequi 0. \]
\end{theorem}

We deduce Theorem \ref{thm:mainII} from the following slightly stronger result:
\begin{theorem} \label{thm:main}
Let $S$ be a noetherian scheme of finite dimension, $\ell$ a prime invertible on $S$, and $E \in \NAlg(\SH(S))$.
Suppose that $\ell^n = 0 \in \pi_0(E)$ for some $n$, and also that \[ E \wedge \H\Z/\ell \wequi 0. \]
Then also \[ E \wedge \MGL/\ell \wequi 0. \]
\end{theorem}

\begin{remark}
In \cite{mathew2017nilpotence}, Mathew, Noel, and Naumann prove the May Nilpotence Conjecture by studying power operations in Morava $E$-theories.  It is beyond the range of present technology to produce a normed spectrum structure on any direct analog of Morava $E$-theory in $\SH(S)$.  Our proof uses, as replacements for Morava $E$-theories, normed spectra $R_n$ that are motivic analogs of
$$\MU[v_n^{-1}]^{\wedge}_{v_0,v_1,\cdots,v_{n-1}}.$$
Since the $R_n$ spectra are built out of $\MGL$ by inverting and completing, but not by quotienting, they are easily seen to be normed spectra.  Other than the change of context from a height $n$ Morava $E$-theory to $R_n$, our proof is similar in spirit to that of Mathew, Noel, and Naumann, though with a few additional complications.  In particular, we must deal with the fact that Morava $K$-theories are not homotopy ring spectra over arbitrary bases $S$.
\end{remark}

\subsection{An \texorpdfstring{$\scr E_{\infty}$ $\MGL$}{E∞ MGL}-algebra with vanishing motive}

To understand the significance of Theorem \ref{thm:mainII}, it is important to note that the assumption that $E$ be normed \emph{cannot} be weakened to $E$ being merely $\scr E_\infty$:

\begin{example}[a non-zero, torsion $\scr E_\infty$ $\MGL$-algebra with vanishing motive]
Let $k = \C$.
Denote by $K(n)^\mot$ the \emph{motivic Morava $K$-theory spectrum}, which we can define as \[ K(n)^\mot = \MGL_{(\ell)}[t_{\ell^n-1}^{-1}] \otimes \bigotimes_{i \ne \ell^n - 1} \MGL_{(\ell)}/t_i; \] see e.g. \S\ref{subsec:generators} for a definition of the $t_i$.

First, note that $K(n)^\mot \wedge \H\Z \wequi 0$.
Indeed, this is an orientable homotopy ring spectrum whose coefficients $\pi_{2*,*}$ form a ring of characteristic $\ell$ over which the formal group law of $K(n)^\mot$ becomes isomorphic to the additive one, and hence must have infinite height \cite[Lemma A2.2.9]{ravenel1986complex}.
This is only possible if $v_n$ maps to zero; since $v_n$ also maps to a unit we conclude that this must be the zero ring.

Denote by $\1_\ell^\wedge/\tau$ the $\scr E_{\infty}$ ring spectrum that is the \emph{motivic cofiber of $\tau$} \cite{gheorghe2017motivic}.
Note that $K(n)^\mot/\tau \ne 0$ (e.g. since we know its homotopy groups, which are computed below).
We shall construct an $\scr E_\infty$-algebra structure on $K(n)^\mot/\tau$.

There exists an equivalence of categories between certain cellular $\MGL_\ell^\wedge/\tau$-modules and the bounded derived category of $\MU_*$-modules \cite[Theorem 3.7]{gheorghespecial}, which is in fact symmetric monoidal\footnote{One way of seeing this is as follows. The work of Pstragowski identifies cellular motivic spectra over $\C$ with (even) \emph{synthetic spectra}, and this equivalence is symmetric monoidal \cite[Theorem 7.34, Remark 7.35]{pstrkagowski2018synthetic}. The modules over cofiber $\tau$ correspond in the two categories, and are symmetric monoidally equivalent to the stable category of $\MU_*\MU$-comodules \cite[Proposition 4.53, Remark 4.55]{pstrkagowski2018synthetic}. This identifies the symmetric monoidal category of modules of $\MGL/\tau$ with modules over $\MU_*\MU$ in the stable category of $\MU_*\MU$-comodules. This is symmetric monoidally equivalent to the derived category of $\MU_*$-modules, by taking ``extended'' comodules. \iftoggle{draft}{\tom{Not sure if we want to say even more here...}}}.
Let $X$ be a cellular $\MGL_\ell^\wedge/\tau$-module such that $\pi_{**}X$ is concentrated in degrees $(2*,*)$.
Then under this equivalence, $X$ is sent to the $\MU_* = \MGL_{2*,*}$-module $\pi_{2*,*} X$.
We have $\pi_{**} K(n)^\mot/\tau = \Z/\ell[v_n, v_n^{-1}]$ with $|v_n| = (2 \ell^n - 2, \ell^n - 1)$, and consequently $K(n)^\mot/\tau$ corresponds to the $\MU_*$-module $\pi_* K(n) = \Z/\ell[v_n, v_n^{-1}]$.
This admits an obvious commutative $\MU_*$-algebra structure, and hence defines an $\scr E_\infty$-ring in the derived category of $\MU_*$-modules\tombubble{right?}.
The above equivalence transports this to an $\scr E_\infty$-ring structure on $K(n)^\mot/\tau \in \MGL\Mod$.
\end{example}

\subsection{Normed nilpotence, related work, and open questions}

It has long been known that a naive analog of Devinatz--Hopkins--Smith nilpotence fails in motivic stable homotopy theory, even in the category $\SH(\C)$.  The prototypical example of this is that the element $\eta$ is not nilpotent in the homotopy of the unit object $\1 \in \SH(\C)$, despite the fact that $\eta$ has trivial Hurewicz image in $\MGL_{**}$.  In fact, $\1[\eta^{-1}]$ is an $\scr E_{\infty}$-ring object in $\SH(\C)$ that is non-zero, though it becomes trivial after tensoring with $\MGL$.  There remains, however, the intriguing fact that $\1[\eta^{-1}]$ is \emph{not} a normed spectrum in $\SH(\C)$:

\begin{example}[Example 12.11 in \cite{bachmann-norms}]
Suppose $S$ is pro-smooth over a field of characteristic $\ne 2$\tombubble{I think the only assumption needed for this is $1/2 \in S$. (Using some tricks that we didn't have at the time the norms paper was written.) In any case this also follows from \cite{bachmann-splitting}.}, and that $R \in \NAlg(\SH(S))$.  If the multiplication by $\eta$ map is an equivalence on $R$, then $R \simeq 0$.
\end{example}

In light of the above example, it seems reasonable to make the following definition:

\begin{definition}
Suppose $R \in \NAlg(\SH(S))$ and $x \in \pi_{*,*} R$.  We say that $x$ is \emph{normed nilpotent} if, whenever $R \to R'$ is a map of normed spectra such that the multiplication by $x$ map is an equivalence on $R'$, $R' \simeq 0$.
\end{definition}

\begin{question} \label{qst:normed-nilpotence}
For what class of normed spectra $R$ is it true that normed nilpotence in $R$ is detected under the Hurewicz map $R \to R \wedge \H\Z$?
\end{question}

Our work here can be thought of as the beginnings of an answer to the above question.  For example, we have the following:

\begin{theorem}
Let $S$ be a noetherian scheme of finite dimension, and write $\scr S$ for the set of primes not invertible on $S$.  Suppose that $R \in \NAlg(\SH(S))$ is also a homotopy left unital $\MGL[\scr S^{-1}]$-algebra, with no assumed relationship between the normed spectrum structure and the homotopy left unital $\MGL[\scr S^{-1}]$-algebra structure.  Then normed nilpotence in $R$ is detected by the Hurewicz map
$$R \to R \wedge \H\Z[\scr S^{-1}].$$
\end{theorem}

\begin{proof}
Let $x$ be an element of $\pi_{*,*} R$ such that $x$ is normed nilpotent in $R \wedge \H\Z[\scr S^{-1}]$.  Our task is to prove that $x$ is normed nilpotent in $R$.  To do so, suppose that $R \to R'$ is a map of normed spectra such that the multiplication by $x$ map is an equivalence on $R'$.  By assumption $\H\Z[\scr S^{-1}] \wedge R' \simeq 0$, and so by Theorem \ref{thm:mainII} we conclude that $\MGL[\scr S^{-1}] \wedge R' \simeq 0$.  However, the composite map $\MGL[\scr S^{-1}] \to R \to R'$ gives $R'$ a homotopy left unital $\MGL[\scr S^{-1}]$-algebra structure, whence $R'$ is a retract of $R' \wedge \MGL[\scr S^{-1}]$, and hence zero.
\end{proof}

Forthcoming work \cite{bachmann-splitting} of the first author, Elden Elmanto and Jeremiah Heller answers Question \ref{qst:normed-nilpotence} in another case.

\begin{theorem} \label{thm:bachmann-elmanto-heller}
Let $S$ be a scheme with $1/2 \in S$, $R$ a normed spectrum over $S$, and $x \in \pi_{2*+1,*} R$ an element in odd simplicial degree.
Then normed nilpotence of $x$ is detected under the Hurewicz map $$R \to R \wedge \H\Z,$$ and in fact even by the Hurewicz map $$R \to R \wedge \H\mathbb{F}_2.$$
\end{theorem}

\begin{remark}
The above theorem provides a proof, slightly different from that in \cite{bachmann-norms}, that $\eta$ is normed nilpotent in $\1 \in \NAlg(\SH(\C))$.
\end{remark}

\begin{remark}
Possibly after a finite separable extension, the Koszul sign rule will imply that $x^2$ is $2$-torsion in $\pi_{*,*} R$.  Theorem \ref{thm:bachmann-elmanto-heller} arises from a motivic analog of Mahowald's stable homotopy theory theorem that an $\scr E_{2}$-ring with $2=0$ is an $\H\mathbb{F}_2$-algebra \cite[Theorem 4.18]{mathew2017nilpotence}.
\end{remark}

When it comes to answering Question \ref{qst:normed-nilpotence} in general, the authors feel a great deal of progress would be made if one could understand the answer to the following:

\begin{question}
Suppose that $R$ is a normed spectrum such that some power $\tau^{k}$ of $\tau$ is $0$ in $\pi_{*,*}R$.  Must $R$ be null?
\end{question}

\begin{remark}
In stable homotopy theory, the May Nilpotence Conjecture can be viewed as a strong restriction on the mixed characteristic of an $\scr E_{\infty}$-ring.  Specifically, one learns that if $H\mathbb{Q} \wedge R \simeq 0$, then $K(n) \wedge R \simeq 0$ for every Morava $K$-theory.  In \cite{hahn2016bousfield}, the second author observed that, even if $H\mathbb{Q} \wedge R \not\simeq 0$, one may deduce that $K(n) \wedge R \simeq 0$ whenever $K(n-1) \wedge R \simeq 0$.
\end{remark}

\begin{question}
For $R$ a normed spectrum in $\NAlg(\SH(S))$, does $K(n-1) \wedge R \simeq 0$ imply that $K(n) \wedge R \simeq 0$?  Are there similar relationships involving the $K(\beta_{ij})$ of \cite{krause2018motivicperiodicity}?
\end{question}

\subsection{Organization of the paper}
In \S\ref{sec:norm-and-transfer} we define the \emph{universal transfer} and \emph{universal norm} of elements in the homotopy groups of (normed) motivic spectra; the former exists because spectra are built out of infinite loop spaces, and the latter is a multiplicative version of the former that is special to normed spectra.
We establish some of their standard properties, which are largely analogous to the classical case.

In the next two sections, we restrict to the base scheme being a field.
After establishing some preliminary results in \S\ref{sec:completion-locn-MGL}, we study in \S\ref{sec:cohom-homology} the norm and transfer in the very special normed $\MGL$-module $R=R_n$.
Our key result is that any torsion normed $R$-module is zero, which is morally related to Tate vanishing in the $K(n)$-local category.
Along the way we compute the homotopy groups $\pi_{2*,*} R$, the cohomology groups $R^{2*,*} \B\mu_\ell$, and we show that the completed homology $\pi_{**}^\wedge \B\mu_\ell \hat\wedge R$ is finitely generated free.

In the final \S\ref{sec:detection} we first establish some general techniques for showing that a normed spectrum is zero.
For example, we show that this may be tested after an arbitrary separable field extension.
Then we combine all our work to deduce the main theorems.

\subsection{Conventions}
In \S\ref{sec:completion-locn-MGL} and \S\ref{sec:cohom-homology} we fix a base field $k$ of exponential characteristic $e \ne \ell$, and all our spectra will be implicitly $e$-periodized.  We make some additional technical assumptions on $k$ in \S\ref{sec:cohom-homology}, elaborated on at the beginning of that section.

We will use the notion of normed spectrum as set out in \cite{bachmann-norms}.
We deviate from the terminology in this reference in one aspect: we will only deal with ``$\Sm$-normed spectra'', and just call them normed spectra; we denote the category of normed spectra over a scheme $S$ by $\NAlg(\SH(S))$.

Throughout we freely use the language of $\infty$-categories, as set out in \cite{lurie-htt,lurie-ha}.

\subsection{Acknowledgments}
Tom Bachmann was supported by NSF Grant DMS-1906072, and Jeremy Hahn by NSF Grant DMS-1803273.  We thank Peter May for a comment improving the exposition.

\section{Norms and transfers in cohomology}
\label{sec:norm-and-transfer}

If $E \in \SH(S)$, then $E$ represents a cohomology theory on smooth $S$-schemes, with all the usual structures.
For example, if $f: X \to Y \in \Sm_S$ is a finite étale morphism, then there is a transfer map $\tr_{X/Y}: E^0(X) \to E^0(Y)$ (to be contrasted with the pullback map $E(Y) \to E(X)$ which of course also exists), which is an additive homomorphism.
If $E \in \NAlg(\SH(S))$, then the cohomology theory is multiplicative, and in particular there is a norm map $N_{X/Y}: E^0(X) \to E^0(Y)$, which is a morphism of multiplicative monoids.
The definition of these maps only uses that $E$ has ``incoherent'' addition and multiplication maps, i.e. for example $N_{X/Y}$ is obtained using the multiplication map $f_\otimes E|_X \to E|_Y$, and similarly for the transfer.
But $E$ actually carries \emph{coherently commutative} addition and multiplication maps, and we can use this structure to build a more complicated transfer map, essentially by taking the homotopy colimit over all transfers along finite étale extensions of a certain type.
For example, the map $\mathrm{E} \Sigma_n \to \B \Sigma_n$ is the ``universal degree $n$ finite étale morphism'', and so one might expect to build a transfer (and norm) $E^0(\mathrm{E} \Sigma_n) \to E^0(\B\Sigma_n)$.
Since $\mathrm{E} \Sigma_n$ is contractible, thus just takes the form \[ \Tr, \N: E^0(S) \to E^0(\B\Sigma_n). \]

In the current section we implement this idea.

\subsection{Extended powers}
We make use of the motivic extended powers $D_{C_n}$, defined in \cite[\S5]{bachmann-colimits}.
In particular these are functors $D_{C_n}: \SH(S) \to \SH(S)$ satisfying $D_{C_n}(\1) \wequi \Sigma^\infty_+ \B C_{n}$, where $\B C_n$ denotes the \emph{geometric} or \emph{étale} classifying space \cite[\S4]{A1-homotopy-theory}.
If $E \in \NAlg(\SH(S))$, then $E$ comes equipped with multiplication maps $m: D_{C_n}(E) \to D_n(E) \to E$.

We need to be a little bit more precise.
Given $E \in \SH(S)$, the evident $\Sigma_n$-action on $E^{\wedge n}$ refines to a ``genuine motivic'' action, yielding $E^{\wedge \ul n} \in \SH^{\Sigma_n}(E)$ \cite[\S6.1]{bachmann-colimits}.
Let us denote by $E^{\wedge C_n} \in \SH^{C_n}(S)$ the result of forgetting from the $\Sigma_n$-action to the $C_n$-action.
Then we have the formula \[ D_{C_n}(E) = E^{\wedge C_n} \wedge_{C_n} \EE C_n, \] where $\EE C_n$ is the universal motivic space with a free $C_n$-action \cite[\S3.1]{gepner-heller}, and $(\ph) \wedge_{C_n} \EE C_n$ denotes the composite \cite[\S5.2]{gepner-heller} \[ \SH^{C_n}(S) \xrightarrow{\wedge \EE C_n} \SH^{C_n}(S)^\free \xrightarrow{(\ph)/C_n} \SH(S). \]

Let us also recall that in order to have a good six functors formalism for $\SH^{C_n}(\ph)$, one needs to assume that $1/|C_n| \in S$ \cite[\S1.2]{hoyois-equivariant}.
This is needed for example if we want finite étale $S$-schemes (with non-trivial action) to be self-dual in $\SH^{C_n}(S)$.\tombubble{At least I don't know any other proof.}

\subsection{Transfer}
We assume throughout that $1/n \in S$.
\begin{definition}
Let $E \in \SH(S)$ be arbitrary.
\begin{enumerate}
\item The \emph{universal transfer map} $\Sigma^\infty_+ \B C_n \to \1$ is obtained by applying $(\ph) \wedge_{C_n} \EE C_n$ to the transfer map $* \to \Sigma^\infty_+ C_n \in \SH^{C_n}(S)$ (obtained using the Wirthmüller/ambidexterity isomorphism).
\item The \emph{stable transfer for $E$} is the map \[ \Tr: E \to E^{\B C_n} \] obtained by applying $\imap(\ph, E)$ to the universal transfer map.
\item The \emph{restriction for $E$} is the map \[ \Res: E^{\B C_n} \to E \] obtained by applying $\imap(\ph, E)$ to the restriction $* \wequi \B e \to \B C_n$.
\end{enumerate}
\end{definition}
\begin{example}
We are mostly interested in the effect of the transfer on homotopy groups.
There it takes the form \[ \Tr: E^{**} \to E^{**}(\B C_n), \] and the restriction map just becomes the pullback map \[ \Res: E^{**}(\B C_n) \to E^{**}(\B e) \wequi E^{**}. \]
\end{example}

\begin{remark}
The transfer and restriction implicitly depend on the integer $n$, which we suppress in the notation.
\end{remark}

\begin{remark} \label{rmk:general-transfer}
The universal transfer construction yields more generally a map $D_{C_n}(X) \to X^{\wedge n}$ for every $X \in \SH(S)$: take geometric homotopy orbits of the transfer map $X^{C_n} \to X^{C_n} \wedge \Sigma^\infty_+ C_n$.
\end{remark}

\begin{lemma} \label{lemm:transfer-properties}
\begin{enumerate}
\item Restriction and transfer are natural in $E$.
\item $\Tr(x + y) = \Tr(x) + \Tr(y)$.
\item The composite $\Res \circ \Tr: E \to E$ is homotopic to the morphism of multiplication by $n$.
\item Suppose that $E$ is a homotopy associative ring spectrum.
  Then the projection formula holds: for $x \in E^{**}, y \in E^{**}(\B C_n)$ we have $\Tr(x \Res(y)) = \Tr(x)y$.
\end{enumerate}
\end{lemma}
\begin{proof}
(1) is clear by construction, as is (2) since the transfer comes from a map of spectra.

(3) In $\SH^{C_n}(S)$ we have the canonical map $\Sigma^\infty_+ C_n \to \1$ (corresponding to the map of $C_n$-sets $C_n \to *$).
Applying $(\ph) \wedge_{C_n} \EE C_n$ we obtain the restriction $* \to \B C_n$\todo{is that totally obvious?}.
We wish to determine the composite of universal restriction and transfer $\1 \to \Sigma^\infty_+ \B C_n \to \1$; by the above this map is obtained by applying $(\ph) \wedge_{C_n} \EE C_n$ to the composite \[ c: \Sigma^\infty_+ C_n \to \1 \to \Sigma^\infty_+ C_n \in \SH^{C_n}(S). \]
As proved in \cite[Lemma 7.11]{bachmann-MGM} (set $G=\{1\}$), there is an additive symmetric monoidal functor \[ \Sigma^\infty_\fet: \Span(\Sm_S^{C_n}, \fet, \all) \to \SH^{C_n}(S)\] such that the image of the backwards map $* \leftarrow C_n \xrightarrow{\wequi} C_n$ is the transfer map $\1 \to \Sigma^\infty_+ C_n$.
Consequently $c$ is the image of the composite span \[ C_n \xleftarrow{\wequi} C_n \to * \leftarrow C_n \xrightarrow{\wequi} C_n \wequi C_n \leftarrow C_n \times C_n \to C_n, \] which is just given by the sum of the maps $m_a: C_n \to C_n, x \mapsto xa$ for $a \in C_n$.
Since $m_a \times_{C_n} \EE C_n \wequi \id_*$ for every $a$, the result follows.\todo{we had some doubts about this...}

(4) We need to show that the two longest paths in the following diagram represent homotopic maps
\begin{equation*}
\begin{CD}
\Sigma^\infty_+ \B C_n @>{\Delta}>> \Sigma^\infty_+ \B C_n \wedge \Sigma^\infty_+ \B C_n \\
@V{\Tr}VV                           @V{\Tr \wedge \id}VV \\
\1                     @>{\Res}>>   \1 \wedge \Sigma^\infty_+ \B C_n @>{x \wedge y}>> E \wedge E @>m>> E.
\end{CD}
\end{equation*}
For this it suffices to show that the left hand square commutes.
It is obtained by applying $\Sigma^\infty_\fet[(\ph) \wedge_{C_n \times C_n} \EE \Sigma_2]$ to the following diagram in $\Span(\Sm_S^{C_n \times C_n}, \fet, \all)$
\begin{equation*}
\begin{CD}
C_n \times C_n/\Delta @>>> * \times * \\
@VVV                       @VVV       \\
C_n \times C_n        @>>> C_n \times *.
\end{CD}
\end{equation*}
Here $\Delta \subset C_n \times C_n$ denotes the diagonal subgroup, the horizontal maps are the obvious ones (just maps of sets), and the vertical ones are the obvious transfers.
The two composites are represented by the spans \[ C_n \times C_n/\Delta \leftarrow (C_n \times C_n/\Delta) \times (C_n \times *) \to C_n \times * \text{ and } C_n \times C_n/\Delta \leftarrow C_n \times C_n \to C_n \times *. \]
These are easily verified to be isomorphic.\todo{do it?}
\end{proof}

%\begin{remark}
%We are slightly surprised by the appearance of $n$ instead of $n_\epsilon$ in Lemma \ref{lemm:transfer-properties}(3). \TODO{say something more here}
%\end{remark}

\subsection{Norm}
\begin{definition}
Let $E \in \NAlg(\SH(S))$. The \emph{norm map for $E$-cohomology} is the map $\N: E^0 \to E^0(\B C_n)$ which sends $x: \1 \to E$ to the composite \[ \Sigma^\infty_+ \B C_n \wequi D_{C_n}(\1) \xrightarrow{D_{C_n}(x)} D_{C_n}(E) \xrightarrow{m} E. \]
\end{definition}
\begin{remark} \label{rmk:norm-spacelevel}
The norm map is obtained by applying $\pi_0$ to the total power operation $P_{C_n}: E_0 \to \ul{\Map}(\B C_n, E_0)$ of \cite[Example 7.25]{bachmann-norms}.
\end{remark}

\begin{lemma} \label{lemm:DCp-sum}
Let $n=p$ be a prime, and $1/p \in S$.
There exists an equivalence of functors $\SH(S) \times \SH(S) \to \SH(S)$ as follows \[ D_{C_p}(E \vee F) \wequi D_{C_p}(E) \vee D_{C_p}(F) \vee \bigvee_s (E, F)^{\wedge s}. \]
Here the sum at the end is over strings of length $p$ in $\{E,F\}$ containing at least one copy of $E$ and one copy of $F$, up to cyclic permutation, and $(E,F)^{\wedge s}$ denotes the corresponding smash product of $E$s and $F$s.\footnote{Strictly speaking, in order to make this into a well-defined functor, we need to pick a representative for each equivalence class. We tacitly assume that this has been done.}
This decomposition has the following properties.
\begin{enumerate}
\item If $x: E \to E'$ and $y: F \to F'$ are morphisms, then the induced map $D_{C_p}(x \vee y): D_{C_p}(E \vee F) \to D_{C_p}(E' \vee F')$ is given by \[ D_{C_p}(x) \vee D_{C_p}(y) \vee \bigvee_s (x,y)^{\wedge s}. \]
\item If $E=F$ and $\Delta: E \to E \vee E$ is the diagonal map, then the map \[ D_{C_p}(E) \xrightarrow{D_{C_p}(\Delta)} D_{C_p}(E \vee E) \wequi D_{C_p}(E) \vee D_{C_p}(E) \vee \bigvee_s E^{\wedge p} \] is given in components by $(\id, \id, \tr, \tr, \tr, \dots, \tr)$, where $\tr: D_{C_p}(E) \to E^{\wedge p}$ is the generalized universal transfer of Remark \ref{rmk:general-transfer}.
\item If $E=F$ and $\nabla: E \vee E \to E$ is the fold map, then the map \[ D_{C_p}(E) \vee D_{C_p}(E) \vee \bigvee_s E^{\wedge p} \wequi D_{C_p}(E \vee E) \xrightarrow{D_{C_p}(\nabla)} D_{C_p}(E) \] is given in components by $(\id, \id, c, c, \dots, c)$, where $c: E^{\wedge p} \to D_{C_p}(E)$ is the canonical map.
\end{enumerate}
\end{lemma}
\begin{proof}
We have $D_{C_p}(E \vee F) \wequi (E \vee F)^{C_p} \wedge_{C_p} \EE C_p$.
Moreover there is a canonical decomposition \[ (E \vee F)^{C_p} \wequi E^{C_p} \vee F^{C_p} \vee \bigvee_s C_{p+} \wedge (E,F)^{\wedge s} \in \SH^{C_p}(S). \]
This immediately implies (1), and (3) is readily verified.

To prove (2), it suffices to determine the natural transformation \[ (\ph)^{C_p} \xrightarrow{(\nabla)^{C_p}} (\ph \vee \ph)^{C_p} \in \Fun(\SH(S), \SH^{C_p}(S)). \]
The unstable version of this transformation, i.e. a morphism in $\Fun(\Spc^\fet(S), \Spc^{C_p,\fet}(S))$ is readily determined.\tombubble{And this is really all we need, since we will apply (2) for $E=\1$...}
We can make $\Spc^{C_p,\fet}(S)$ into a $\Spc^\fet(S)$-module via $X \mapsto X^{C_p}$; then we have a transformation of $\Spc^{\fet}(S)$-module functors.
Base change along $\Spc^\fet(S) \to \SH(S)$ yields the transformation we were seeking to determine.\todo{details?}
The result follows.
\end{proof}

\begin{lemma} \label{lemm:norm-props}
Let $E \in \NAlg(\SH(S))$.
\begin{enumerate}
\item The norm is natural in $E \in \NAlg(\SH(S))$.
\item We have $\N(1) = 1$ and $\N(0) = 0$, where we denote by $1$ the multiplicative unit in both $E^0$ and $E^0(\B C_n)$.
\item For $x, y \in E^0$ we have $\N(xy) = \N(x)\N(y)$.
\item For $x, y \in E^0$ and $n=p$ a prime with $1/p \in S$ we have \[ \N(x + y) = \N(x) + \N(y) + \Tr\left(\sum_{i \le j, i + j=p} \frac{{p \choose i}}{p} x^i y^j \right).\]
\end{enumerate}
\end{lemma}
\begin{proof}
(1), (2) and (3) are clear (e.g. combine Remark \ref{rmk:norm-spacelevel} with \cite[\S7.2]{bachmann-norms}).
(4) is just a decategorification of Lemma \ref{lemm:DCp-sum}.
\end{proof}

\begin{corollary} \label{cor:norm-trick}
Let $n=p$ a prime with $1/p \in S$.
Then \[ \N(p^k) \equiv -p^{k-1} \Tr(1) \pmod{p^k} \in \pi_0(E). \]
\end{corollary}
\begin{proof}
Mirroring the proof of Lemma \ref{lemm:DCp-sum} and Lemma \ref{lemm:norm-props}(4), given any integer $m \ge 1$ we can decompose the $C_p$-set $[m]^{C_p}$ into $m$ points with trivial $C_p$-action and $(m^p-m)/p$ copies of $C_p$.
This tells us that $N(m) = m + (m^p - m)/p\Tr(1)$.
(Note that $m^p-m$ \emph{is} divisible by $p$.)
The result follows immediately.\tombubble{Maybe it would make more sense to derive this from Lemma \ref{lemm:norm-props}(4) directly...}
\end{proof}

\section{Completions, localizations, and quotients of \texorpdfstring{$\MGL$}{MGL}}
\label{sec:completion-locn-MGL}
We collect some preliminary results.
Starting from \S\ref{subsec:pi*-MGL}, $k$ is a field of exponential characteristic $e \ne \ell$, and all objects are implicitly $e$-periodized.

\subsection{Localization and completion of normed spectra}

\begin{proposition}
Let $k$ be a field, $E \in \NAlg(\SH(k))$ and $\alpha: E \to L$ be an $E$-module map for some $L \in \pi_0 Pic(E\Mod) = \Gamma$.
\begin{enumerate}
\item Suppose that for every finite separable field extension $L/K$ over $k$, the element $\N_{L/K}(\alpha_L)$ is a unit in the $\Gamma$-graded commutative ring $\pi_\star(E_K)[\alpha_K^{-1}]$.
  Then $E[\alpha^{-1}]$ is a normed spectrum.
\item Suppose that for every finite separable field extension $L/K$ over $k$, the element $\alpha_K$ is a unit in the $\Gamma$-graded commutative ring $\pi_\star(E_K)[\N_{L/K}(\alpha_L)^{-1}]$.
  Then $E_\alpha^\wedge$ is a normed spectrum.
\end{enumerate}
\end{proposition}
\begin{proof}
Examination of the proofs of \cite[Proposition 12.6, Proposition 12.14]{bachmann-norms} shows that it suffices to establish the following: for every finite (surjective) étale morphism $f: X \to Y \in \Sm_k$, the functor $f_\otimes: E_X\Mod \to E_Y\Mod$ preserves $\alpha$-periodic equivalences (respectively $\alpha$-complete equivalences).
Using \cite[Corollary 14]{bachmann-real-etale} and compatibility of norms with base change \cite[Proposition 5.3]{bachmann-norms}, for this we may assume that $Y$ is the spectrum of a field.
Then the desired result follows from our assumption, exactly as in the proofs of \cite[Proposition 12.6, Proposition 12.14]{bachmann-norms}.
\end{proof}

\begin{remark}
If $E$ is oriented and $L = T^n \wedge E$ for some $n$, then $\pi_{f_\otimes L}(E) \wequi \pi_{T^{[L:K]}}(E)$ canonically, and the notion of a $\Gamma$-graded commutative ring can be dispensed with in favor of the ordinary commutative graded ring $\pi_{2*,*}(E)$.
\end{remark}
\begin{example} \label{ex:locn-trivial}
If $\pi_{2*,*}(E)[\alpha_K^{-1}] = \pi_{2*,*}(\N_{L/K}(\alpha_L)^{-1})$, then the Proposition applies and $E[\alpha^{-1}]$, $E_\alpha^\wedge$ are normed spectra.
This happens in particular if $\N_{L/K}(\alpha_L) = \alpha_K^{[L:K]}$.
\end{example}

\subsection{The homotopy of MGL} \label{subsec:pi*-MGL}
Recall our convention that the exponential characteristic of $k$ has been inverted throughout; without this assumption the following result is (at the time of writing) not known.
\begin{lemma}[Spitzweck] \label{lemm:MGL-vanishing}
We have $\MGL_{2q,q}(K) = L_{2q}$ (where $L_{2q} = \pi_{2q} \MU$), $\MGL_{2q+1,q}(K) = K^\times \otimes L_{2q+2}$ and $\pi_{p,q}(\MGL)(K) = 0$ for $p < 2q$ (and also for $p<q$).
\end{lemma}
\begin{proof}
See \cite[Proposition 7.1, Corollary 7.4 and Corollary 7.5]{SpitzweckMGL}.
\end{proof}

\subsection{Localizations and completions of $\MGL$}
\begin{corollary} \label{cor:loc-MGL}
Let $a_1, \dots, a_r, b_1, \dots, b_s \in \MGL_{2*,*}(k)$ and $E \in \NAlg(\MGL\Mod)$.
Then \[ E[b_1^{-1}, \dots, b_s^{-1}]_{a_1, \dots, a_r}^\wedge \] is a normed spectrum.

Similarly with $\MGL_{(\ell)}$ in place of $\MGL$.\tombubble{Not sure what the slickest way of formulating a general statement is.}
\end{corollary}
\begin{proof}
Since tensoring with an invertible object preserves limits (and colimits), $b_i$-periodic spectra are stable under limits and colimits.
Since also $a_j$-complete spectra are stable under limits, it suffices to treat the periodization or completion at one element $u$.
By Example \ref{ex:locn-trivial}, it suffices to show that $\N_{L/K}(u) = u^{[L:K]}$.
Consider the presheaf $\scr F = \MGL_{2*,*}(\ph)$, say on $\FEt_K$, viewed as a presheaf of multiplicative monoids.
It follows from Lemma \ref{lemm:MGL-vanishing} that this is a constant sheaf, and hence in particular an étale sheaf.
We deduce from this and \cite[Corollary C.13]{bachmann-norms} that there is a unique normed structure on $\scr F$ compatible with the multiplication.
Since both the norms coming from $\MGL \in \NAlg(\SH(k))$ and the norms coming from the assignment $\N_{L/K}(a) = a^{[L:K]}$ are compatible with the multiplication in $\scr F$, they must agree.
The case of $\MGL_{(\ell)}$ is completely analogous.
This concludes the proof.\tombubble{I feel that it should be possible to say this better.}
\end{proof}

\subsection{\texorpdfstring{$\MGL_{(\ell)}$, $\BPGL$,}{BPGL} and generators}
\label{subsec:generators}

Recall that there is a retraction of homotopy commutative rings \cite[Definition 5.3]{vezzosi2001brown} \[\BPGL \to \MGL_{(\ell)} \to \BPGL. \]
We shall fix for all time an orientation of $\BPGL$, which therefore also gives an orientation of $\MGL_{(\ell)}$.
This orientation induces an isomorphism of rings \cite[Proposition 3.5]{vezzosi2001brown} \[ \BPGL^{2*,*}(\P^{\infty}) \cong \BPGL_{2*,*}\fpsr{x}, \] and the latter ring comes equipped with a formal group law.

\begin{proposition} \label{prop:BP-homotopy}
It is possible to choose elements $v_i \in \pi_{2\ell^i-2,\ell^i-1} \BPGL$ such that:
\begin{enumerate}
\item The map $$\Z_{(\ell)}[v_1,v_2,\cdots] \to \pi_{2*,*} \BPGL$$
is an isomorphism of rings.
\item We have the formula, in $\BPGL^{2*,*}(\P^{\infty})$,
$$[\ell](x) \equiv v_i x^{\ell^i} \text{ modulo }\ell,v_1,\cdots,v_{i-1},x^{\ell^i+1}.$$
\end{enumerate}
\end{proposition}
Here and throughout, $[\ell](x)$ denotes the $\ell$-series of the formal group law \cite[Definition A2.1.19]{ravenel1986complex}.
\begin{proof}
The formal group law on $\pi_{2*,*}\MGL$ induces $\MU_{2*} \to \pi_{2*,*} \MGL$, which is an isomorphism by Lemma \ref{lemm:MGL-vanishing}.
By construction, the same holds for $\pi_{2*} \BP \to \pi_{2*,*} \BPGL$, so these claims reduce to their classical analogs, which are well-known \todo{reference?}.
\end{proof}

We denote by $t_{\ell^i-1}$ the image of $v_i$ in $\pi_{2\ell^i-2,\ell^i-1} MGL_{(\ell)}$.

\begin{proposition} \label{prop:MU-homotopy}
It is possible to choose, for all $i \ne \ell^n-1$, generators $t_i \in \pi_{2i,i} MGL_{(\ell)}$ such that:
\begin{enumerate}
\item $\pi_{2*,*}(MGL_{(\ell)}) \cong \Z_{(\ell)}[t_1,t_2,\cdots],$ and
\item $h(t_i) \equiv b_i$ modulo decomposables ($i \ne \ell^n-1$), where $h: \pi_{**} \MGL_{(\ell)} \to \H\Z_{**} \MGL_{(\ell)} \wequi \Z_{(\ell)}[b_1, b_2, \dots]$ is the Hurewicz map.
\end{enumerate}
\end{proposition}
\begin{proof}
Note that for (2), it suffices to satisfy the stronger analogous condition for the $\MGL$-Hurewicz map $\pi_{**}\MGL_{(\ell)} \to \MGL_{**} \MGL_{(\ell)}$\tombubble{right?}.
The formal group law on $\pi_{2*,*} \MGL_{(\ell)}$ induces a morphism of Hopf algebroids \[ (\pi_{2*}\MU_{(\ell)}, \MU_{2*}\MU_{(\ell)}) \to (\pi_{2*,*} \MGL_{(\ell)}, \MGL_{2*,*}\MGL_{(\ell)} \] preserving the $b_i$ on both sides (see e.g. \cite[p. 23]{hoyois-algebraic-cobordism}).
This reduces the claims (in the strengthened form) to their classical analogs, which are well-known\todo{reference?}.
\end{proof}

Let $\scr C$ be a presentably symmetric monoidal $\infty$-category and $X \in \scr C$.
We can form the free $A_\infty$-ring on $X$ \cite[Proposition 4.1.1.14]{lurie-ha}, denoted by \[ \1_{\scr C}[X] \wequi \bigvee_{n \ge 0} X^{\wedge n}. \]
This ring is homotopy commutative if and only if the switch map on $X \wedge X$ is homotopic to the identity\tombubble{right?}.
This applies in particular if $\scr C = \MGL\Mod$ and $X = \MGL(i)[j]$, provided that $j$ is \emph{even}.
We may thus form the homotopy commutative ring and map of homotopy commutative rings \begin{equation} \label{eq:BPGL-MGL} \BPGL[T_i \mid i \ne \ell^n-1] := \BPGL \otimes_{\MGL} \bigotimes_i \MGL[T_i] \to \MGL_{(\ell)}, \end{equation} where $T_i$ is in degree $(2i,i)$ and is sent to $t_i \in \pi_{**} \MGL_{(\ell)}$.
\begin{lemma} \label{lemm:BPGL-MGL}
The map \eqref{eq:BPGL-MGL} is an equivalence.
\end{lemma}
\begin{proof}
Let us denote the map by $\alpha$.
Since $\alpha$ is a map of very effective homotopy $\MGL_{(\ell)}$-modules and $\ul{\pi}_0^\eff(\MGL_{(\ell)}) = \H\Z_{(\ell)}$, it suffices to show that we $\alpha \wedge \H\Z_{(\ell)}$ is an equivalence.
Since $\MGL \wedge \H\Z = \H\Z[b_1, b_2 \dots]$ is a sum of of objects of the form $\H\Z(i)[2i]$ for $i \in \Z$, it suffices to show that $\pi_{2i,i}(\alpha \wedge \H\Z_{(\ell)})$ is an isomorphism for each $i \in \Z$.
Since $\pi_{2*,*}(\MGL \wedge \H\Z) \wequi \pi_{2*}(\MU \wedge \H\Z)$ and similarly for $\BPGL$, our claim reduces to the classical analog\todo{really?}, which is well-known\todo{reference?}.
\end{proof}

\begin{lemma} \label{lemm:hopkins-morel-modified}
We have
\begin{enumerate}
\item $\MGL_{(\ell)}/(t_1, t_2, \dots) \wequi \H\Z_{(\ell)}$,
\item $\MGL_{(\ell)}/(t_i \mid i \ne \ell^n-1) \wequi \BPGL$,
\item $\BPGL/(v_1, v_2, \dots) \wequi \H\Z_{(\ell)}$.
\item $\MGL_{(\ell)}/(v_1, v_2, \dots) \wequi \H\Z_{(\ell)}[T_i \mid i \ne \ell^n-1]$.
\end{enumerate}
\end{lemma}
\begin{remark}
Note that a priori we only obtain $\BPGL$ as an object of $\SH(k)$, not $\MGL\Mod$.
In particular (3) does not really make sense.
However (2) exhibits a lift of $\BPGL$ to $\MGL\Mod$, allowing us to make sense of the iterated (infinite) quotient.
\end{remark}
\begin{proof}
(1) This is a minor adaptation of the main theorem of \cite{hoyois-algebraic-cobordism}.
It suffices to show that we have an equivalence after $\wedge \H\Q$ and $\wedge \H\Z/\ell$. The arguments of \cite[Lemma 7.8 and Lemma 7.9]{hoyois-algebraic-cobordism} go through essentially unchanged (the key fact we use here is that our generators $t_{\ell^n-1}$ are \emph{adequate} in the sense of \cite[Definition 7.1]{hoyois-algebraic-cobordism}).\tombubble{Is there a slicker way of seeing this?}

(2) We claim that the composite $\beta: \BPGL \to \MGL_{(\ell)} \to \MGL_{(\ell)}/(t_i \mid i \ne \ell^n-1)$ is an equivalence.
As in the proof of Lemma \ref{lemm:BPGL-MGL}, it suffices to prove that $\beta \wedge \H\Z_{(\ell)}$ is an equivalence.
By construction, $\H\Z_{**} \MGL_{(\ell)}/(t_i \mid i \ne \ell^n-1)$ is a polynomial ring on generators in the same bidegrees as the $v_i$\tombubble{say something clearer here?}.
This implies that $\H\Z \wedge \MGL_{(\ell)}/(t_i \mid i \ne \ell^n-1)$ is a sum of objects of the form $\H\Z_{(\ell)}(i)[2i]$ for $i \in \Z$.
It hence suffices to check that $\H\Z_{2*,*} \beta$ is an equivalence, whence the claim reduces to the classical analog, which is well-known\todo{reference?}.

(3) Follows from (1) and (2) since the formation of the quotients is independent of the order of elements being killed.

(4) Follows from (3) and Lemma \ref{lemm:BPGL-MGL}.
\end{proof}

\subsection{Ring structures on quotients of $\MGL$}
\begin{definition}
We call a field $k$ \emph{$\ell$-good} if $\cd_\ell(k) = 0$ and $k$ contains all $\ell^n$-th roots of unity for all $n$.
\end{definition}
\begin{lemma} \label{lemm:MGL-homotopy}
Suppose that $k$ is $\ell$-good.
Then \[ \pi_{**} \MGL_{\ell}^\wedge \wequi L_{\ell}^\wedge[\tau] \text{ and } \pi_{**} \MGL/\ell^n \wequi L/\ell^n[\tau], \] where $|\tau|=(0,-1)$.
\end{lemma}
\begin{proof}
The first statement follows from the second (for all $n$).
Our assumptions imply that $H^{**}(k, \Z/\ell^n) = \Z/\ell^n[\tau]$; this follows from the Bloch-Kato conjecture \cite[Theorem 6.17]{voevodsky2011motivic}\tombubble{And various easier results...}.
Noting that $s_q(\MGL) = \Sigma^{2q,q} \H\Z \otimes L_{2q}$ \cite[Corollary 4.7]{spitzweck2010relations} (see also \cite[Theorem 3.1]{SpitzweckMGL}), the result follows from the strongly convergent slice spectral sequence for $\MGL/\ell^n$ \cite[Theorem 8.12]{hoyois-algebraic-cobordism}, which degenerates at the $E_1$-page without extension problems, since everything is concentrated in even degrees and every degree only has elements in a unique filtration.
\end{proof}

\begin{definition} \label{def:homotopy-left-unital}
Let $\scr C$ be a symmetric monoidal category.
We say that $E \in \scr C_{\1/}$ \emph{admits a homotopy left unital multiplication} if there exists $m: E \otimes E \to E$ such that the following diagram commutes
\begin{equation*}
\begin{CD}
\1 \otimes E @>>> E \otimes E \\
   @|                      @V{m}VV        \\
E            @=    E.
\end{CD}
\end{equation*}
\end{definition}

In the lemma below, we work with $\scr C$ the category of $MGL$-modules:

\begin{lemma} \label{lemm:final-2}
Let $x_1, \dots, x_n \in \pi_{2*,*} \MGL_{\ell}^\wedge$.
\begin{enumerate}
\item Suppose that $k$ is $\ell$-good.
  Then $\MGL_{\ell}^\wedge/(x_1, \dots, x_n)$ admits a homotopy left unital $\MGL$-algebra structure.
\item Suppose that $x_1 = \ell^N$.
  Then there exists a finite separable extension $k'/k$ such that $\MGL_{\ell}^\wedge/(x_1, \dots, x_n)|_{k'}$ admits a homotopy left unital $\MGL$-algebra structure.
\end{enumerate}
\end{lemma}
\begin{proof}
(1) It suffices to deal with $n=1$.
Note that by Lemma \ref{lemm:MGL-homotopy}, $x=x_1$ is not a zero-divisor in $\pi_{**} \MGL_{\ell}^\wedge$; in particular both $\MGL_\ell^\wedge$ and $\MGL_\ell^\wedge/x$ have homotopy groups concentrated in even (simplicial) degrees.
We can now argue exactly as in \cite[Lemma 3.2, Corollary 3.3, Lemma 3.4]{strickland1999products}: the evenness of $\MGL_\ell^\wedge$ implies that $[\MGL_\ell^\wedge/x, \MGL_\ell^\wedge/x]_{2*,*} \xrightarrow{\alpha} [\MGL_\ell^\wedge, \MGL_\ell^\wedge/x]_{2*,*}$ is an injection, and in particular $x$ acts by zero on the left hand side (since it does so on the right hand side).
It follows that $\MGL_\ell^\wedge/x \wedge \MGL_\ell^\wedge/x \wequi \MGL_\ell^\wedge/x \vee \Sigma \MGL_\ell^\wedge/x$ (non-canonically), and we can take a projection to the first summand as the multiplication map.
Injectivity of $\alpha$ implies that any such multiplication is unital (even on both sides), as desired.

(2) Let $R = \MGL_{\ell}^\wedge/(x_1, \dots, x_n)$.
We need to find $m: R \otimes R \to R$.
Let $k'/k$ be an algebraic separable $\ell$-good extension (e.g. a separable closure of $k$\tombubble{is this ok?}).
Then $m': R \otimes R|_{k'} \to R|_{k'}$ exists, by our assumptions and (1).
Since $R \otimes R$ is a compact $\MGL_{(\ell)}$-module, by continuity \cite[Lemma A.7(1)]{hoyois-algebraic-cobordism} there exists a finite subextension $k'/l/k$ and $m: R \otimes R|_l \to R|_l$ such that $m|_{k'} \wequi m'$.
Replacing $k$ by $l$ if necessary, we may thus assume that $m$ exists.
We need $m$ to satisfy certain properties.
They are expressed as commutativity of certain further diagrams of compact $\MGL_{(\ell)}$-modules, and hold for $m|_{k'}$ by construction.
Hence using continuity again, they hold after passing to a possible further finite subextension.
This concludes the proof.
\end{proof}

\begin{example}
The finite extension cannot be avoided in general.  For example, if $k=\R$, then not every quotient $\MGL_{\ell}^{\wedge}/t_i$ can admit a homotopy left unital $\MGL$-algebra structure.  If they all did, then their $C_2$-equivariant realizations would all be homotopy $\MU_{\mathbb{R}}$-algebras, and it would follow that the Real Morava $K$-theories were $C_2$-equivariant homotopy $\MU_{\R}$-algebras.  That this cannot be was pointed out by Hu-Kriz \cite[p. 332]{hukriz2001real}.

In general, as the second author first learned from Xiaolin Danny Shi, it is impossible for $\BP_{\R}/v_i$ to be a homotopy ring when $i>0$.  If this were possible, then the Postnikov truncation map $\BP_{\mathbb{R}}/v_i \to \H\underline{\mathbb{Z}_{(2)}} \to \H\underline{\mathbb{F}_2}$ would induce a map of ordinary rings
$$\pi_*(\Phi(\BP_{\mathbb{R}}/v_i)) \to \pi_*(\Phi(\H\underline{\mathbb{F}_2})).$$
Here, $\Phi$ is the $C_2$ geometric fixed points functor.  The codomain of this map is a polynomial ring $\mathbb{F}_2[a]$, but the image consists of just the unit and $a^{2^i}$.  In particular, the image is not a subring of the codomain.
\end{example}

\section{The transfer over an adequate field} \label{sec:cohom-homology}

Consider the spectrum $\MGL_{(\ell)} \in \SH(k)$.  By \cite[Proposition 12.8]{bachmann-norms}, this is a normed spectrum, and by Lemma \ref{lemm:MGL-vanishing} we have $\pi_{2*,*} \MGL_{(\ell)} \cong L_* \otimes \Z_{(\ell)}$.  In Proposition \ref{prop:MU-homotopy} we chose generators of this latter ring such that
$$\pi_{2*,*} \MGL_{(\ell)} \cong \Z_{(\ell)}[t_1,t_2,\cdots].$$
The map $\pi_{2*,*}\BPGL \to \pi_{2*,*}\MGL_{(\ell)}$ sends $v_i$ to $t_{p^i-1}$.  When referring to elements of $\pi_{2*,*} \MGL_{(\ell)}$, we will use $v_i$ and $t_{p^i-1}$ interchangeably.

\begin{definition} \label{def:adequate}
Fix for the remainder of this section an integer $n>0$.  We call a field $k$ \emph{adequate} if in $\SH(k)$, the $\MGL$-module $$\MGL/(\ell,v_1,\cdots,v_{n-1})$$ admits the structure of a homotopy left unital $\MGL$-algebra.
\end{definition}

Note that, by Lemma \ref{lemm:final-2}, every field $k$ admits a finite separable extension $k'$ that is adequate.
Moreover $\ell$-good fields (such as separably closed fields) are adequate, and so are extensions of adequate fields.

\begin{convention}
We will assume for the remainder of this section that the field $k$ is adequate.
\end{convention}

\begin{definition}
We define the normed spectrum $R$ (depending on the implicit integer $n$) to be
$$R=\MGL_{(\ell)}[v_n^{-1}]^{\wedge}_{(\ell,v_1,v_2,\cdots,v_{n-1})}.$$
Note that $R$ is a normed spectrum by Corollary \ref{cor:loc-MGL}.
\end{definition}

Our goal will be to understand the completed $R$-module
$$(R \wedge \B\mu_{\ell})^{\wedge}_{(\ell,v_1,v_2,\cdots,v_{n-1})}.$$
We will see that it is a free, finitely-generated $R$-module, and we will name a basis for it in terms of elements in
$$R^{2*,*}(\B \mu_{\ell}).$$
If the field $k$ has a primitive $\ell$th root of unity, there is an equivalence of classifying spaces $\B \mu_{\ell} \simeq \B C_\ell$.
We will end by giving (under this assumption) a formula for the stable transfer
$$R^{2*,*} \to R^{2*,*}(\B C_{\ell}) \cong R^{2*,*}(\B \mu_{\ell}).$$

\subsection{The homotopy of \texorpdfstring{$R$}{R}}

\begin{definition} \label{def:adequate}
For a tuple of non-negative integers $K = (k_0, k_1, \dots, k_{n-1})$, we define \[ M_K = R/(\ell^{k_0}, v_1^{k_1}, \dots, v_{n-1}^{k_{n-1}}). \]
We will use $M$ to denote
$$M=M_{(1,1,\cdots,1)} \simeq R/(\ell,v_1,\dots,v_{n-1}) \simeq \MGL[v_{n}^{-1}]/(\ell,v_1,\dots,v_{n-1}).$$
\end{definition}
Since $$M \simeq \MGL/(\ell,v_1,\dots,v_{n-1}) \otimes_{\MGL} R,$$ the assumption that $k$ is an adequate field ensures that $M$ is a homotopy left unital $R$-algebra.

\begin{lemma} \label{lemm:vanising-start}
Suppose $K$ is a tuple of non-negative integers.  Then
$$\pi_{2*,*}M_{K} \cong (\pi_{2*,*} \MGL)[v_n^{-1}]/(\ell^{k_0},v_1^{k_1},\dots,v_{n-1}^{k_{n-1}}),$$
and $\pi_{p,q} M_K = 0$ for $p < 2q$.  Furthermore, if $K'$ is another tuple of integers such that $k'_i \ge k_i$ for all $0 \le i \le n-1$, then
the map $\pi_{2*+1,*} M_{K'} \to \pi_{2*+1,*} M_K$ is surjective.
\end{lemma}

\begin{proof}
For $0 \le i \le n-1$, let $A_i$ denote $\MGL[v_n^{-1}]/(\ell^{k_0},\dots,v_i^{k_{i}})$.  Then $A_0=\MGL[v_n^{-1}]$ and $A_{n-1}=M_K$.
We understand all the relevant homotopy groups of $A_0$ by Lemma \ref{lemm:MGL-vanishing}, and we may inductively understand the homotopy groups of $A_{i+1}$
via the long exact sequence associated to the cofiber sequence
$$\Sigma^{k_{i+1}|v_{i+1}|} A_{i} \to A_i \to A_{i+1}.$$
To be explicit, we have the long exact sequence
$$\pi_{(p,q)+(\ell^{i+1}-1)(2,1)} A_i \xrightarrow{v_i^{k_i}} \pi_{p,q} A_i \to \pi_{p,q} A_{i+1} \to \pi_{(p-1,q)+(\ell^{i+1}-1)(2,1)} A_i \xrightarrow{v_i^{k_i}} \pi_{p-1,q} A_i.$$
The lemma follows from the inductive calculations
$$\pi_{p,q} A_i \cong 0 \text{ if }p<2q,$$
$$\pi_{2*,*} A_i \cong L[v_n^{-1}]_{2*}/(\ell^{k_0},v_1^{k_1},\dots,v_i^{k_{i}}), \text{ and}$$
$$\pi_{2*+1,*} A_i \cong (k^{\times}/\ell^{k_0}) \otimes L[v_n^{-1}]_{2*+2}/(v_1^{k_1},\dots,v_i^{k_i}).$$ \jeremybubble{Can you verify this last bit? \tom{Looks right.}}
\end{proof}

\begin{proposition} \label{prop:homotopy-of-R}
We have the formula
$$\pi_{2*,*} R \cong (\pi_{2*,*} \MGL[v_n^{-1}])^{\wedge}_{\ell,v_1,v_2,\cdots,v_{n-1}}.$$
In particular, none of $\ell,v_1,\cdots,v_{n-1}$ are $0$-divisors inside of $\pi_{2*,*} R$.
\end{proposition}

\begin{proof}
By definition, $R$ is the homotopy limit of the $M_K$.  The Milnor exact sequence \cite[Proposition VI.2.15]{goerss2009simplicial} thus says that
$$0 \to \text{lim$^1$} \pi_{2*+1,*} M_K \to \pi_{2*,*} R \to \lim_{K} \pi_{2*,*} M_K \to 0,$$
and the final clause of Lemma \ref{lemm:vanising-start} guarantees that the lim$^1$ term vanishes.
\end{proof}

The following useful lemma is one of the main reasons we assumed the field $k$ to be adequate.
\begin{lemma} \label{lemm:trivial-after-tensor}
Let $x \in \pi_{2*,*} R$ be an element of the ideal $(\ell,v_1,v_2,\cdots,v_{n-1})$.  Then the $R$-module homomorphism
$$x:\Sigma^{|x|} R \to R$$
becomes trivial after tensoring down to $M$.  In other words, the $R$-module and homotopy $M$-module map
$$\Sigma^{|x|} M \simeq \Sigma^{|x|} R \otimes_{R} M \xrightarrow{x} R \otimes_{R} M \simeq M$$
is nullhomotopic.
\end{lemma}

\begin{proof}
This map is a map of homotopy $M$-modules, and hence determined by a class in $\pi_{|x|} M$.  Thus, it suffices to check that $x$ maps to $0$ under the unit map
$$\pi_{2*,*} R \to \pi_{2*,*} M.$$
This is clear by construction.
\end{proof}

\subsection{The completed \texorpdfstring{$R$}{R}-homology\texorpdfstring{\text{ of }$\B \mu_\ell$}{}}
Our chosen orientation of $\BPGL$ induces an orientation of $R$, giving an isomorphism $$R^{**}(\P^\infty) \wequi R^{**}\fpsr{x}$$
as well as a formal group law.  We denote by $[\ell](x) \in R^{**}\fpsr{x}$ the $\ell$-series of this formal group law.  Note that 
$$[\ell](x)=\ell x + \cdots$$
is a power series in $x$, congruent to $\ell x$ modulo $x^2$, and with coefficients in $\BPGL_{2*,*}$.

\begin{lemma} \label{lemm:mul-cohomology-2}
There is a cofiber sequence of $\MGL$-modules \[ \B \mu_{\ell+} \wedge \MGL \to \P^\infty \wedge \MGL \xrightarrow{[\ell]} \P^\infty \wedge \MGL(1)[2]. \]
The map $\P^\infty \wedge \MGL \xrightarrow{[\ell]} \P^\infty \wedge \MGL(1)[2]$ is determined by its dual, which in $\MGL^{**}(\P^{\infty})$ is multiplication by $[\ell](x)$.
\end{lemma}
\begin{proof}
The geometric construction of $\B \mu_\ell$ shows that this space can be obtained as $\scr O(-\ell)_{\P^\infty} \setminus 0$ \cite[Lemma 6.3]{voevodsky2003reduced}, and consequently there is a cofiber sequence \[ \B \mu_\ell \wequi \scr O(-\ell)_{\P^\infty} \setminus 0 \to \scr O(-\ell)_{\P^\infty} \to Th(\scr O(-\ell)_{\P^\infty}). \]
Smashing with $\MGL$ and using the Thom isomorphism we obtain a cofiber sequence of $\MGL$-modules \[ \B \mu_{\ell+} \wedge \MGL \to \P^\infty \wedge \MGL \to \P^\infty \wedge \MGL(1)[2]. \]
The final statement comes from the relation of the Euler class of $\scr O(-\ell)$ with the $\ell$-series in the formal group law, which is essentially the definition of the formal group law.
\end{proof}

\begin{lemma} \label{lemm:homology-start}
There is an equivalence of $R$-modules \[ (R \wedge \Sigma^\infty_+ \B \mu_\ell)_{\ell, v_1, \dots, v_{n-1}}^\wedge \wequi \bigvee_{k=0}^{\ell^n-1} \Sigma^{2k,k} R. \]
\end{lemma}

\begin{proof}
Lemma \ref{lemm:mul-cohomology-2} implies that there is a cofiber sequence \[ \B \mu_{\ell+} \wedge R \to \P^\infty \wedge R \xrightarrow{[\ell]} \P^\infty \wedge R(1)[2]. \]
Here $\P^\infty \wedge R \wequi \bigvee_{i \ge 0} R\{\beta_i\}$, where $\beta_i$ in bidegree $(2i,i)$ is dual to $x^i$, and the map $[\ell]$ is dual to multiplication by the power series $[\ell](x)$.
Let $M$ denote $M_{(1,1,\dots,1)}$.  The key property of our preferred generators $t_i$ of $L_* \otimes \Z_{(\ell)}$ is that \[ [\ell](x) \equiv v_n x^{\ell^n} \!\mod(x^{\ell^n + 1}, \ell, v_1,v_2,\cdots,v_{n-1}). \]
This implies, using Lemma \ref{lemm:trivial-after-tensor}, that $[\ell] \otimes_R M$ is the homotopy $M$-module map specified by
\begin{align*}
  \beta_0 &\mapsto 0 \\
  \beta_1 &\mapsto 0 \\
  \dots \\
  \beta_{\ell^n-1} &\mapsto 0 \\
  \beta_{\ell^n} &\mapsto v_{n} \beta_0 \\
  \beta_{\ell^n+1} &\mapsto v_{n} \beta_1 + ?\beta_0 \\
  \beta_{\ell^n+2} &\mapsto v_{n} \beta_2 + ?\beta_1 + ?\beta_0 \\
  \dots
\end{align*}
Since $v_n$ is a unit in $R$, it follows that if we split $M \wedge \P^\infty = P_1 \vee P_2$, where $P_1 = \bigvee_{k=0}^{\ell^n-1} M\{\beta_k\}$ and $P_2 = \bigvee_{k \ge \ell^n} M\{\beta_k\}$, then the restriction of $[\ell] \otimes_{R} M$ to $P_1$ is zero whereas the restriction to $P_2$ is an isomorphism $P_2 \to M \wedge \P^\infty$.  In particular, if we similarly set $R \wedge \P^{\infty} = Q_1 \vee Q_2$, where $Q_1 = \bigvee_{k=0}^{\ell^n-1} R$ and $Q_2=\bigvee_{k \ge \ell^n} R$, then the restriction of $[\ell]$ to $Q_2$ is an equivalence onto $R \wedge \P^{\infty}$ after completion at the ideal $(\ell,v_1,\cdots,v_{n-1})$.  It follows that
$$(R \wedge \Sigma^\infty_+ \B \mu_\ell)_{\ell, v_1, \dots, v_{n-1}}^\wedge \wequi Q_1.$$
\end{proof}

\subsection{The \texorpdfstring{$R$}{R}-cohomology\texorpdfstring{\text{ of }$\B \mu_\ell$}{}}
%Since $\ell \in k^\times$ there is a short exact sequence of étale sheaves of groups \[ 1 \to \mu_\ell \to \Gm \to \Gm, \] inducing a fiber sequence of étale classifying spaces
%\begin{equation} \label{eq:fiber-seq}
%\B_\et \mu_\ell \to \B_\et \Gm \to \B_\et \Gm.
%\end{equation}
%By homotopy invariance of the groups of line bundles and units on smooth varieties over a field \cite{TODO}, $\B_\et \Gm$ is motivically local and hence so are all spaces in \eqref{eq:fiber-seq}.
%Consequently this is a fiber sequence of motivic spaces (or see \cite[Theorem 2.2.5]{asok2015affine} for a more general statement)
%\begin{equation} \label{eq:fiber-seq-mot}
%\Gm \to \B \mu_\ell \to \P^\infty \to \P^\infty.
%\end{equation}

We calculate $R^{2*,*}(\B \mu_{\ell})$ by copying the classical argument for $E^*(\B C_\ell)$, where $E$ is a complex-oriented ring spectrum \cite[\S 5.4]{hopkins2000characters}.

\begin{lemma} \label{lemm:weierstrass}
In the ring $R^{2*,*}\fpsr{x}$, we have $[\ell](x) = u g(x)$, where $u$ is a unit and \[ g(x) = a_{\ell^n} x^{\ell^n} + a_{\ell^{n-1}}x^{\ell^n - 1} + \dots + \ell x, \] for some coefficients $a_i \in R^{2*,*}$ such that $a_{\ell^n}$ is a unit.
\end{lemma}
\begin{proof}
The key property of our preferred generators $t_i$ of $L_* \otimes \Z_{(\ell)}$ is that \[ [\ell](x) \equiv t_{\ell^n - 1} x^{\ell^n} \!\mod(x^{\ell^n + 1}, \ell, t_1, \dots, t_{\ell^n-2}). \]
Since $t_{\ell^n-1}$ is a unit in $R_{2*,*}$, we may apply the Weierstrass Preparation Theorem \cite[Theorem 2.10]{omalley1972weierstrass} to the power series $\frac{[\ell](x)}{x}$.  This yields an expression
$$\frac{[\ell](x)}{x} = h(x) v,$$
where $v$ is a unit in $R^{2*,*}\fpsr{x}$ and $h(x)$ is a monic degree $\ell^n$ polynomial with coefficients in $R^{2*,*}$ and constant term $a_{\ell^n}^{-1} \ell$ for some unit $a_{\ell^n} \in R^{2*,*}$.
We define $g(x)= a_{\ell^n} xh(x)$ and $u = a_{\ell^n}^{-1}v$.
This concludes the proof.
\end{proof}

\begin{corollary} \label{cor:cohomologyanswer}
There are canonical isomorphisms
$$R^{2*,*}(\B\mu_{\ell}) \cong R^{2*,*}\fpsr{x}/[\ell](x) \cong R^{2*,*}[x]/g(x).$$
In particular, $R^{2*,*}(\B\mu_{\ell})$ is a finite free $R^{**}$-module with basis $1, x, \dots, x^{\ell^n-1}$.
\end{corollary}

\begin{proof}
The first statement follows from Lemma \ref{lemm:mul-cohomology-2} via the long-exact sequence
$$R^{2*,*}(\Sigma^{-1,0} \P^{\infty}) \longleftarrow R^{2*,*}(\B\mu_{\ell}) \longleftarrow R^{2*,*}(\P^\infty) \stackrel{[\ell](x)}{\longleftarrow} R^{2*,*}(Th(\scr O(-\ell)_{\P^\infty})).$$
The key point is that $R^{2*,*}(\Sigma^{-1,0} \P^{\infty}) \cong 0$ by Lemma \ref{lemm:MGL-vanishing}.
The second statement then follows from Lemma \ref{lemm:weierstrass}.
\end{proof}

\subsection{The transfer and its section}
For this section, we assume furthermore that $k$ contains a primitive $\ell$-th root of unity.
Thus $\mu_\ell \wequi C_\ell$.

Our first goal will be to determine a formula for the stable transfer $$R \to R^{\B C_\ell}.$$  Specifically, applying $\pi_{2*,*}$ gives an $R^{2*,*}$-linear map $$\text{Tr}(\--):R^{2*,*} \to R^{2*,*}(\B C_{\ell}) \cong R^{2*,*}(\B\mu_{\ell}) \cong R^{2*,*}\fpsr{x}/[\ell](x) \cong R^{2*,*}[x]/g(x),$$
and we seek a formula for this homomorphism.

\begin{lemma}
The stable transfer map
$$\text{Tr}(\--):R^{2*,*} \to R^{2*,*}[x]/g(x)$$
is the unique $R^{2*,*}$-linear homomorphism such that
$$\text{Tr}(1) = \frac{g(x)}{x}.$$
\end{lemma}

\begin{proof}
The argument is exactly analogous to that in \cite[6.15]{hopkins2000characters}; we recall it for the convenience of the reader.
Note that the restriction map $R^{2*,*}(\B C_\ell) \to R^{2*,*}(Be)$ is the $R^{2*,*}$-algebra homomorphism sending $x$ to $0$.   The projection formula of Lemma \ref{lemm:transfer-properties} thus implies that $x \Tr(1) = 0$.
Corollary \ref{cor:cohomologyanswer} implies via a straightforward calculation that the only classes in $R^{2*,*}(\B C_{\ell})$ that are killed by $x$ are the $R^{2*,*}$-multiples of $\frac{g(x)}{x}$.
Thus $$\Tr(1)=r \frac{g(x)}{x},$$ for some element $r \in R^{2*,*}$.
Since the constant term of $g(x)/x$ is $\ell$, Lemma \ref{lemm:transfer-properties}(3) implies that $r\ell=\ell$.
By Proposition \ref{prop:homotopy-of-R}, $\ell$ is not a $0$-divisor in $\pi_{2*,*} R$, and it follows that $r=1$.
This concludes the proof.
\end{proof}

Recall that the stable transfer arises from a map
$$\Sigma^{\infty}_+ \B C_\ell \to \1.$$
After tensoring this map with $R$, we obtain a map
$$R \wedge \Sigma^{\infty}_+ \B C_\ell \to R.$$
Since the codomain $R$ is complete, we may complete the domain of the above to obtain
$$(R \wedge \Sigma^{\infty}_+ \B C_\ell)^{\wedge}_{\ell,v_1,\cdots,v_{n-1}} \to R.$$

\begin{corollary} \label{cor:gamma}
There is an $R$-module homomorphism 
$$\gamma:R \to (R \wedge \Sigma^{\infty}_+ \B C_\ell)^{\wedge}_{\ell,v_1,\cdots,v_{n-1}}$$
such that the composite of $\gamma$ with the transfer
$$(R \wedge \Sigma^{\infty}_+ \B C_\ell)^{\wedge}_{\ell,v_1,\cdots,v_{n-1}} \to R$$
is the identity on $R$.
\end{corollary}

\begin{proof}
Since $(R \wedge \B C_\ell)^{\wedge}_{\ell,v_1,\cdots,v_{n-1}}$ is a finitely generated free $R$-module, we may specify the map $\gamma$ by specifying its $R$-linear dual.  This dual element is specified by a class in $R^{2*,*}[x]/g(x)$, which we know to be a free $R^{2*,*}$-module generated by $1,x,\cdots,x^{\ell^n-1}$.
Recall that $g(x)=a_{\ell^n} x^{\ell^n} + a_{\ell^{n}-1} x^{\ell^n-1} + \cdots \ell x$, where $a_{\ell^n}$ is a unit in $R^{2*,*}$.
We choose $\gamma$ to be dual to the class $a_{\ell^n}^{-1} x^{\ell^n-1}$.  The result then follows from the formula $$\Tr(1)=\frac{g(x)}{x},$$ the fact that the coefficient of $x^{\ell^n-1}$ in $\frac{g(x)}{x}$ is $a_{\ell^n}$, and the meaning of specifying a map via its dual\tombubble{I would like to say something clearer here, maybe.}.
\end{proof}

The following is the main result of this section.
\begin{corollary} \label{cor:vanishing-smash-product}
Let $k$ be an adequate field (see Definition \ref{def:adequate}) of exponential characteristic $e \ne \ell$ containing a primitive $\ell$-th root of unity.

Let $A \in \NAlg(\SH(S))_{R/}$, and assume that $\ell^k = 0 \in \pi_0(A)$ for some $k \ge 0$.
Then $A \wequi 0$.
\end{corollary}
\begin{proof}
If $k=0$, then this is just the statement that a ring $A$ is zero if and only if $1=0$ in $\pi_{0}A$.
In general, it will suffice by induction to prove that $\ell^{k-1} = 0 \in \pi_0(A)$.
Using Corollary \ref{cor:norm-trick} (and naturality of norms) we find that \[ 0 = \N(\ell^k) = -\ell^{k-1} \Tr(1) \in A^0(\B C_\ell). \]
This implies that the composite
$$\Sigma^{\infty}_+ \B C_{\ell} \xrightarrow{\Tr(1)} \1 \to R \stackrel{\ell^{k-1}}{\to} A$$
is nullhomotopic, where the last map is $\ell^{k-1}$ times the unit map.  Using the fact that $R$ is complete, we obtain that the composite of $R$-module maps
$$(R \wedge \Sigma^{\infty}_+ \B C_\ell)^{\wedge}_{\ell,v_1,\cdots,v_{n-1}} \xrightarrow{\Tr(1)} R \stackrel{\ell^{k-1}}{\to} A$$
is null.  Consequently \[ 0 = \ell^{k-1} \Tr(1) \gamma = \ell^{k-1} \in A^0, \] making use of Corollary \ref{cor:gamma}.
\end{proof}

\section{Detecting zero rings} \label{sec:detection}
\subsection{Detection in abstract \texorpdfstring{$\infty$-}{∞-}categories}

Recall the notion of a homotopy left unital ring from Definition \ref{def:homotopy-left-unital}.
\begin{lemma} \label{lemm:detect-zero-ring}
Let $E \in \scr C$ carry a homotopy left unital multiplication.
Then $E \wequi 0$ if and only if the unit $1: \1 \to E$ is null homotopic.
\end{lemma}
\begin{proof}
By definition, the identity map on $E$ factors as $E \wequi \1 \wedge E \xrightarrow{1 \wedge \id} E \wedge E \xrightarrow{m} E$ and hence is null homotopic as soon as $1$ is.
The converse is clear.
\end{proof}

\begin{lemma} \label{lemm:detect-abstract}
Let $\scr C$ be a stable symmetric monoidal $\infty$-category which has filtered colimits compatible with $\otimes$.
Let $\{L_i \in \scr C\}_i$ be invertible objects and $\{x_i: L_i^{-1} \to \1\}$ maps.
Put
\begin{itemize}
\item $A_n := \1/(x_1, \dots, x_n) = \1/x_1 \otimes \dots \otimes \1/x_n$;
\item $A_\infty := \1/(x_1, x_2, \dots) = \colim_n \1/(x_1, \dots, x_n)$;
\item $B_{n+1} := \1/(x_1, \dots, x_n)[x_{n+1}^{-1}] = \1/(x_1, \dots, x_n) \otimes \1[x_{n+1}^{-1}]$.
\end{itemize}

Let $E \in \scr C$.
\begin{enumerate}
\item Let $M \ge N \ge 0$ such that (i) $E \wedge A_M \wequi 0$ and (ii) $E \wedge B_{n+1} \wequi 0$ for $n = N, N+1, \dots, M-1$.
  Then $E \wedge A_N \wequi 0$.
\item Suppose that $\1 \in \scr C$ is compact and both $E$ and each $A_n$ (for $n \ge N$) admit a homotopy left unital multiplication.
  Suppose further that (i) $E \wedge A_\infty \wequi 0$ and (ii) $E \wedge B_{n+1} \wequi 0$ for $n = N, N+1, N+2, \dots$.
  Then $E \wedge A_N \wequi 0$.
\end{enumerate}
\end{lemma}
\begin{proof}
(1) Let us show that $E \wedge A_{n+1} \wequi 0$ and $E \wedge B_{n+1} \wequi 0$ together imply $E \wedge A_n \wequi 0$.
This will imply the result by induction.
We have the cofiber sequence \[ L_{n+1}^{-1} \wedge E \wedge A_{n} \xrightarrow{x_{n+1} \wedge \id_E \wedge \id_{A_{n}}} E \wedge A_{n} \to \1/x_{n+1} \wedge E \wedge A_{n} \wequi E \wedge A_{n+1} \wequi 0, \] from which we conclude that $x_n \wedge \id_E \wedge \id_{A_{n}}$ is an equivalence.
By definition, $E \wedge B_{n+1} = E \wedge A_{n} \wedge \1[x_{n+1}^{-1}]$ is obtained as the filtered colimit of \[ E \wedge A_{n} \xrightarrow{x_{n+1}' \wedge \id_E \wedge \id_{A_{n}}} L_{n+1} \wedge E \wedge A_{n} \xrightarrow{\id_{L_{n+1}} \wedge x_{n+1}' \wedge \id_E \wedge \id_{A_{n}}} L_{n+1}^{\otimes 2} \wedge E \wedge A_{n} \to \dots, \] where \[ x_{n+1}' = \id_{L_{n+1}} \wedge x_{n+1}: \1 \wequi L_{n+1} \wedge L_{n+1}^{-1} \to L_{n+1} \wedge \1 \wequi L_{n+1}. \]
Consequently this is a filtered colimit along equivalences, and the colimit (which is $0$ by assumption) is equivalent to any one object.
Hence $E \wedge A_n \wequi 0$, as desired.

(2) Let $C, C' \in \scr C$ have homotopy left unital multiplications.
Then so does $C \otimes C'$.
If $C_1, C_2, \dots \in \scr C$ have homotopy left unital multiplications then so does $C_\infty := \colim_i C_i$.
In particular, $C_\infty \wequi 0$ if and only if $C_n \wequi 0$ for some $n$ (using that $\1$ is compact and Lemma \ref{lemm:detect-abstract}).

We deduce that $A_n, A_\infty, E \wedge A_n, E \wedge A_\infty$ have homotopy left unital multiplications, and $E \wedge A_n \wequi 0$ for $n$ sufficiently large.
If $n \le N$ then the unit $\1 \to E \wedge A_N$ factors as $\1 \to E \wedge A_n \to E \wedge A_N$.
Since the middle term is zero we conclude that $E \wedge A_N \wequi 0$ as desired.
We may thus assume that $n > N$.
We have reduced to (1).
This concludes the proof.
\end{proof}

\subsection{Detection in $\MGL$-modules}
In this subsection we assume again that all spectra are $e$-periodized, where $\ell \ne e$ is the exponential characteristic of $k$.
\begin{lemma} \label{lemm:final}
Let $E \in \SH(k)$ be a homotopy left unital ring spectrum with $E \wedge \H\Z/\ell = 0$ and also $E \wedge \MGL[v_{i+1}^{-1}]/(\ell, v_1, \dots, v_i) = 0$ for all $i \ge 0$.
\begin{enumerate}
\item If $k$ is $\ell$-good, then $E \wedge \MGL/\ell \wequi 0$.
\item In general, there exists a finite separable extension $k'/k$ such that $E \wedge \MGL/\ell|_{k'} \wequi 0$.
\end{enumerate}
\end{lemma}
\begin{proof}
By Lemma \ref{lemm:hopkins-morel-modified}(4) we know that $E \wedge \MGL/(\ell, v_1, \dots) \wequi 0$.
It follows that the unit map $\1 \to E \wedge \MGL/(\ell, v_1, \dots, v_N)$ is null, for $N$ sufficiently large.
By Lemma \ref{lemm:final-2}, possibly after replacing $k$ by a finite separable extension, we may assume that $\MGL/(\ell, v_1, \dots, v_N)$ has a homotopy left unital multiplication.
Hence by Lemma \ref{lemm:detect-zero-ring} we know that $E \wedge \MGL/(\ell, v_1, \dots, v_N) \wequi 0$.
Thus we conclude by Lemma \ref{lemm:detect-abstract}(1).
\end{proof}

\subsection{Detecting zero normed spectra}
\begin{lemma} \label{lemm:detect-zero-normed}
Let $K/k$ be a (not necessarily algebraic) separable field extension and $E \in \NAlg(\SH(k))$ with $E|_{K} \wequi 0$.
Then $E \wequi 0$.
\end{lemma}
\begin{proof}
It is enough to show that $1=0 \in \pi_0(E)$.
Since the functor $K \mapsto \pi_0(E|_K)$ preserves filtered colimits \cite[Lemma A.7(1)]{hoyois-algebraic-cobordism}, we may assume that $K/k$ is finitely generated.
We may further reduce (by induction) to either a simple purely transcendental extension or a finite separable extension.
For the simple purely transcendental case, we use that \[ \pi_0(E) = \ul{\pi}_{0,0}(E)(k) \wequi \ul{\pi}_{0,0}(E)(k[x]) \hookrightarrow \ul{\pi}_{0,0}(E)(k(x)) = \pi_0(E|_{k(x)}), \] i.e. homotopy invariance and unramifiedness of homotopy sheaves of motivic spectra \cite[Theorem 6.2.7 and Lemma 6.4.4]{morel2005stable}.
Finally if $l/k$ is finite separable and $E|_l \wequi 0$ then the norm $N_{l/k}(E|_l) \to E$ \cite[\S7.2]{bachmann-norms} is a ring map with vanishing source and hence vanishing target.
This concludes the proof.
\end{proof}

\subsection{Proof of the main theorem}
We first prove a version of our main theorem over fields.
\begin{theorem} \label{thm:mainfinal}
Let $k$ be a field, $\ell$ a prime invertible in $k$, and $E \in \NAlg(\SH(S))$.
Suppose that $\ell^n = 0 \in \pi_0(E)$ for some $n$, and also that \[ E \wedge \H\Z/\ell \wequi 0. \]
Then also \[ E \wedge \MGL/\ell \wequi 0. \]
\end{theorem}
\begin{proof}
It is necessary and sufficient to show that $(E \wedge \MGL)_\ell^\wedge \wequi 0$.
Let \[ R_i = \MGL_{(\ell)}[v_i^{-1}]^{\wedge}_{(\ell,v_1,v_2,\cdots,v_{i-1})}. \]
For each $i$, there exists a finite separable extension $k_i/k$ containing a primitive $\ell$-th root of unity which is adequate (in the sense of Definition \ref{def:adequate}) for $R_i$.
Since base change along a finite separable extension commutes with completion (being a right adjoint), $R_i$ is stable under finite separable base change, and hence\NB{Or: $R_i \wedge E|_{k_i}$ is complete since $E$ is torsion.} $R_i \wedge E|_{k_i} \wequi 0$ by Corollary \ref{cor:vanishing-smash-product}.
Thus by Lemma \ref{lemm:detect-zero-normed} we find that $R_i \wedge E \wequi 0$, and hence $\MGL[v_{i+1}^{-1}]/(\ell, v_1, \dots, v_i) \wedge E \wequi 0$ for all $i$.
By Lemma \ref{lemm:final}, there exists a further finite separable extension $k'/k$ such that $E \wedge \MGL/\ell|_{k'} \wequi 0$, whence $(E \wedge \MGL)_\ell^\wedge|_{k'} \wequi 0$.
We conclude by a final application of Lemma \ref{lemm:detect-zero-normed}.
\end{proof}

One may prove that normed spectra are stable under arbitrary base change; this is explained in the forthcoming work \cite{bachmann-splitting}.
Using this, the above easily implies our result over more general bases.
\begin{theorem}
Let $S$ be a noetherian scheme of finite dimension, $\ell$ a prime invertible on $S$, and $E \in \NAlg(\SH(S))$.
Suppose that $\ell^n = 0 \in \pi_0(E)$ for some $n$, and also that \[ E \wedge \H\Z/\ell \wequi 0. \]
Then also \[ E \wedge \MGL/\ell \wequi 0. \]
\end{theorem}
\begin{proof}
Pullback to fields is conservative \cite[Corollary 14]{bachmann-real-etale}, formation of normed spectra is compatible with arbitrary base change \cite{bachmann-splitting}, and $\MGL$ is stable under base change (since the Grassmannians it is built out of are).
\end{proof}

\begin{theorem} \label{thm:mainIIfinal}
Let $S$ be a noetherian scheme of finite dimension, and write $\scr S$ for the set of primes not invertible on $S$.
Let $E \in \NAlg(\SH(S))$ and suppose that \[ E \wedge \H\Z[\scr S^{-1}] \wequi 0. \]
Then also \[ E \wedge \MGL[\scr S^{-1}] \wequi 0. \]
\end{theorem}
\begin{proof}
It suffices to show that for every $\ell \not\in \scr S$ we have $E \wedge \MGL_{(\ell)} \wequi 0$, whence replacing $E$ by $E \wedge \MGL_{(\ell)}$ we may assume that $E$ is an $\ell$-local $\MGL$-module.
Under this assumption $E[1/\ell] \wequi E_\Q \wequi E \wedge \H\Q \wequi 0$, and hence $\ell^n = 0 \in \pi_0(E)$ for some $n$.
Since also $E \wedge \H\Z/\ell \wequi 0$, the theorem implies that $E \wedge \MGL/\ell \wequi 0$.
From this we deduce that also $E \wedge \MGL/\ell^N \wequi 0$ for all $N$.
Finally since $\ell^n = 0$ in $E$, $E \wedge \MGL$ is a wedge summand of $E \wedge \MGL/\ell^n$, which is zero.
This concludes the proof.
\end{proof}

\bibliographystyle{alpha}
\bibliography{bibliography}

\begin{thebibliography}{MNN17}

\bibitem[Bac18]{bachmann-real-etale}
Tom Bachmann.
\newblock Motivic and real étale stable homotopy theory.
\newblock {\em Compositio Mathematica}, 154(5):883–917, 2018.
\newblock \href{https://arxiv.org/abs/1608.08855}{arXiv:1608.08855}.

\bibitem[Bac22]{bachmann-MGM}
Tom Bachmann.
\newblock Motivic spectral mackey functors.
\newblock \href{https://arxiv.org/abs/2205.13926}{arXiv:2205.13926}, 2022.

\bibitem[BEH21]{bachmann-colimits}
Tom Bachmann, Elden Elmanto, and Jeremiah Heller.
\newblock Motivic colimits and extended powers.
\newblock \href{https://arxiv.org/abs/2104.01057}{arXiv:2104.01057}, 2021.

\bibitem[BEH22]{bachmann-splitting}
Tom Bachmann, Elden Elmanto, and Jeremiah Heller.
\newblock Splitting results for normed spectra.
\newblock in preparation, 2022.

\bibitem[BH21]{bachmann-norms}
Tom Bachmann and Marc Hoyois.
\newblock Norms in motivic homotopy theory.
\newblock {\em Ast{\'e}risque}, 425, 2021.
\newblock \href{https://arxiv.org/abs/1711.03061}{arXiv:1711.03061}.

\bibitem[DHS88]{devinatz1988nilpotence}
Ethan~S Devinatz, Michael~J Hopkins, and Jeffrey~H Smith.
\newblock Nilpotence and stable homotopy theory i.
\newblock {\em Annals of Mathematics}, 128(2):207--241, 1988.

\bibitem[GH19]{gepner-heller}
David Gepner and Jeremiah Heller.
\newblock The tom dieck splitting theorem in equivariant motivic homotopy
  theory.
\newblock {\em Journal of the Institute of Mathematics of Jussieu}, pages
  1--70, 2019.

\bibitem[Ghe17]{gheorghe2017motivic}
Bogdan Gheorghe.
\newblock The motivic cofiber of $\tau$.
\newblock {\em arXiv preprint arXiv:1701.04877}, 2017.

\bibitem[GJ09]{goerss2009simplicial}
Paul~G Goerss and John~F Jardine.
\newblock {\em Simplicial homotopy theory}.
\newblock Springer Science \& Business Media, 2009.

\bibitem[GWX]{gheorghespecial}
Bogdan Gheorghe, Guozhen Wang, and Zhouli Xu.
\newblock The special fiber of the motivic deformation of the stable homotopy
  category is algebraic (2018).
\newblock arXiv:1809.09290.

\bibitem[Hah16]{hahn2016bousfield}
Jeremy Hahn.
\newblock On the $\mathrm{B}$ousfield classes of $\mathrm{H}_{\infty}$-ring
  spectra.
\newblock {\em arXiv preprint arXiv:1612.04386}, 2016.

\bibitem[HK01]{hukriz2001real}
Po~Hu and Igor Kriz.
\newblock Real-oriented homotopy theory and an analogue of the
  {A}dams-{N}ovikov spectral sequence.
\newblock {\em Topology}, 40(2):317--399, 2001.

\bibitem[HKR00]{hopkins2000characters}
Michael~J. Hopkins, Nicholas~J. Kuhn, and Douglas~C. Ravenel.
\newblock Generalized group characters and complex oriented cohomology
  theories.
\newblock {\em J. Amer. Math. Soc.}, 13(3):553--594, 2000.

\bibitem[Hoy15]{hoyois-algebraic-cobordism}
Marc Hoyois.
\newblock From algebraic cobordism to motivic cohomology.
\newblock {\em Journal f{\"u}r die reine und angewandte Mathematik (Crelles
  Journal)}, 2015(702):173--226, 2015.

\bibitem[Hoy17]{hoyois-equivariant}
Marc Hoyois.
\newblock The six operations in equivariant motivic homotopy theory.
\newblock {\em Advances in Mathematics}, 305:197--279, 2017.

\bibitem[Kra18]{krause2018motivicperiodicity}
Achim Krause.
\newblock Periodicity in motivic homotopy theory and over
  $\mathrm{BP}_*\mathrm{BP}$.
\newblock {\em PhD thesis, Max Planck Institute for Mathematics}, 2018.

\bibitem[Lur09]{lurie-htt}
Jacob Lurie.
\newblock {\em Higher topos theory}.
\newblock Number 170. Princeton University Press, 2009.

\bibitem[Lur16]{lurie-ha}
Jacob Lurie.
\newblock Higher algebra, May 2016.

\bibitem[MNN17]{mathew2017nilpotence}
A.~Mathew, N.~Naumann, and J.~Noel.
\newblock Nilpotence and descent in equivariant stable homotopy theory.
\newblock {\em Advances in Mathematics}, 305:994--1084, 2017.

\bibitem[Mor05]{morel2005stable}
Fabien Morel.
\newblock The stable $\mathbb{A}^1$-connectivity theorems.
\newblock {\em K-theory}, 35(1):1--68, 2005.

\bibitem[MV99]{A1-homotopy-theory}
Fabien Morel and Vladimir Voevodsky.
\newblock $\mathbb{A}^1$-homotopy theory of schemes.
\newblock {\em Publications Mathématiques de l'Institut des Hautes Études
  Scientifiques}, 90(1):45--143, 1999.

\bibitem[O'M72]{omalley1972weierstrass}
Matthew O'Malley.
\newblock On the weierstrass preparation theorem.
\newblock {\em The Rocky Mountain Journal of Mathematics}, 2(2):265--273, 1972.

\bibitem[Pst18]{pstrkagowski2018synthetic}
Piotr Pstragowski.
\newblock Synthetic spectra and the cellular motivic category.
\newblock {\em arXiv preprint arXiv:1803.01804}, 2018.

\bibitem[Rav86]{ravenel1986complex}
Douglas~C Ravenel.
\newblock {\em Complex cobordism and stable homotopy groups of spheres}, volume
  121.
\newblock Academic press New York, 1986.

\bibitem[Spi10]{spitzweck2010relations}
Markus Spitzweck.
\newblock Relations between slices and quotients of the algebraic cobordism
  spectrum.
\newblock {\em Homology, Homotopy and Applications}, 12(2):335--351, 2010.

\bibitem[Spi12]{spitzweck2012commutative}
Markus Spitzweck.
\newblock A commutative $\mathbb{P}^1$-spectrum representing motivic cohomology
  over dedekind domains.
\newblock {\em arXiv preprint arXiv:1207.4078}, 2012.

\bibitem[Spi14]{SpitzweckMGL}
Markus Spitzweck.
\newblock Algebraic cobordism in mixed characteristic.
\newblock 2014.

\bibitem[Str99]{strickland1999products}
Neil Strickland.
\newblock Products on $\mathrm{MU}$-modules.
\newblock {\em Transactions of the American Mathematical Society},
  351(7):2569--2606, 1999.

\bibitem[Vez01]{vezzosi2001brown}
Gabriele Vezzosi.
\newblock Brown-peterson spectra in stable $\mathbb{A}^1$-homotopy theory.
\newblock {\em Rendiconti del Seminario Matematico della Universit{\`a} di
  Padova}, 106:47--64, 2001.

\bibitem[Voe03]{voevodsky2003reduced}
Vladimir Voevodsky.
\newblock Reduced power operations in motivic cohomology.
\newblock {\em Publications Math{\'e}matiques de l'Institut des Hautes
  {\'E}tudes Scientifiques}, 98:1--57, 2003.

\bibitem[Voe11]{voevodsky2011motivic}
Vladimir Voevodsky.
\newblock On motivic cohomology with $\mathbb{Z}/l$-coefficients.
\newblock {\em Annals of mathematics}, 174(1):401--438, 2011.

\end{thebibliography}
\end{document}